\documentclass[12pt]{amsart}
\usepackage{amsfonts, amsmath, amssymb, amsthm}
\usepackage[colorlinks=true,citecolor=blue]{hyperref}
\usepackage{txfonts}
\usepackage[T1]{fontenc}
\usepackage{bbm}
\usepackage{mathtools}

\usepackage{graphicx}
\usepackage{enumerate}
\usepackage{verbatim}
\usepackage{tikz}
\usepackage{algorithmic}

\usetikzlibrary{graphs,graphs.standard}

\usetikzlibrary{decorations,decorations.pathreplacing}

\begingroup
    \makeatletter
    \@for\theoremstyle:=definition,remark,plain\do{%
        \expandafter\g@addto@macro\csname th@\theoremstyle\endcsname{%
            \addtolength\thm@preskip\parskip
            }%
        }
\endgroup

\newtheorem{theorem}{Theorem}
\newtheorem*{theorem*}{Theorem}
\newtheorem{lemma}[theorem]{Lemma}

\newtheorem{corollary}[theorem]{Corollary}

\theoremstyle{definition}

\newtheorem*{remark*}{Remark}

\newtheorem*{problem*}{Problem}

\newcommand{\symg}[1]{   S \hspace{-.4ex}_{#1} } 

\newcommand{\simplex}[2]{ \Delta^{#1}( #2 )} 

\newcommand{\vertex}[2]{V^{(#1)}{(#2)}} 

\newcommand{\face}[2]{\sigma^{(#1)}{(#2)}} 

\newcommand{\bary}[1]{\mathrm{sd}(#1)}

\begin{document} 

\title{A topological product Tverberg Theorem}  

\author[A.~F.~Holmsen]{Andreas F. Holmsen$^\dagger$}\thanks{$^\dagger$Department of Mathematical Sciences, KAIST, Daejeon, South Korea \and Discrete Mathematics Group, Institute for Basic Sciences (IBS), Daejeon, South Korea. \texttt{andreash@kaist.edu}. Supported by the Institute for Basic Science (IBS-R029-C1).}

\author[G.~McCourt]{Grace McCourt$^\ast$}\thanks{$^\ast$Department of Mathematics, Iowa State University, USA. \texttt{gmccourt@iastate.edu}.}

\author[D.~McGinnis]{Daniel McGinnis$^\ddagger$}\thanks{$^\ddagger$Department of Mathematics, Princeton University, USA.  \texttt{dm7932@princeton.edu}. Supported by NSF award no. 2402145.}

\author[S.~Zerbib]{Shira Zerbib$^\diamond$} \thanks{$^\diamond$Department of Mathematics, Iowa State University, USA. \texttt{zerbib@iastate.edu}. Supported by NSF CAREER award no. 2336239  and Simons Foundation award no. MP-TSM-00002629.}

\date{\today}


\begin{abstract} 
We prove a  generalization of the topological Tverberg theorem. One special instance of our general theorem is the following:
Let $\Delta$ denote the 8-dimensional simplex viewed as an abstract simplicial complex, and suppose that its vertices are arranged in a $3\times 3$ array. Then for any continuous map $f:\Delta \to \mathbb{R}^3$ it is possible to partition the rows or the columns of the vertex array into two parts, such that the disjoint faces $\sigma$ and $\tau$ induced by the two parts satisfy $f(\sigma)\cap f(\tau) \neq \emptyset$.  Our result also has consequences for geometric transversals and topological Helly theorems. 
\end{abstract}

\maketitle

\section{Introduction}

\subsection{Background}
Tverberg's theorem \cite{Tve66} is a famous  result of combinatorial geometry asserting that any set of $(d+1)(p-1)+1$ points in $\mathbb{R}^d$ can be partitioned into $p$ parts such that their convex hulls have a point in common. 
This theorem has been extremely influential, with a number of applications and generalizations, and was the starting point of several research directions within geometric and topological combinatorics. See e.g. \cite{Bara-Sob} for a recent survey. 

Here is a standard reformulation of Tverberg's theorem. Let $X$ be a set of $n$ points in $\mathbb{R}^d$, and let $\Delta(n)$ denote the $(n-1)$-dimensional simplex with vertex set $V$. 
(Note that here, somewhat unconventionally, $n$ denotes the number of vertices and {\em not} the dimension.)
The point set $X$ ``induces'' a unique {\em affine} map $f : \Delta(n) \to \mathbb{R}^d$, with $f(V) = X$, which maps a face of $\Delta$ to the convex hull of the corresponding points in $X$. From this point of view, Tverberg's theorem states that {\em for any affine map} \[f: \Delta(n) \to \mathbb{R}^d\] with $n=(d+1)(p-1)+1$, there exists {\em pairwise disjoint} faces $\sigma_1, \sigma_2, \dots, \sigma_p$ of $\Delta(n)$ such that 
\[f(\sigma_1) \cap f(\sigma_2) \cap \cdots \cap f(\sigma_p) \neq \emptyset.\]

B{\'a}r{\'a}ny, Shlosman, and Sz{\"u}cs \cite{BSS} gave a topological generalization of Tverberg's theorem by showing that the statement above, in the case when $p$ is prime, is in fact true for any {\em continuous map} $f$. Their result was further extended by Volovikov \cite{volov96} who showed that the theorem also holds when $p$ is a prime power. It was a long-standing open problem whether or not the topological Tverberg theorem holds for arbitrary integers $p$, which was eventually disproved by Frick \cite{F}.

\smallskip

The goal of this note is to introduce a new type of generalization of the topological Tverberg theorem.

\subsection{Main result} We start by fixing some notation. 
For positive integers $n$ and $m$, 
let $\simplex{m}{n}$ denote the $(n^m-1)$-dimensional simplex on the vertex set 
\[V = [n]\times [n] \times \cdots \times [n] = [n]^m.\]
Note that $\simplex{1}{n} = \Delta(n)$ is the $(n-1)$-dimensional simplex. 
For $i \in [m]$ and a subset $A\subset [n]$ we define the vertex subset
\[\vertex{i}{A} = \big\{ \: (x_1, x_2, \dots, x_n) \in [n]^m : x_i \in A\: \big\} \subset V.\]
In particular, we have $\vertex{i}{\emptyset} = \emptyset$. 
Observe that a fixed $i \in [m]$ and partition $A_1 \cup A_2 \cup \cdots \cup A_p = [n]$, induces a partition of the vertices of $\simplex{m}{n}$
\[ \vertex{i}{A_1} \cup \vertex{i}{A_2} \cup \cdots \cup \vertex{i}{A_p} = V.\]
Finally, for $i \in [m]$ and a subset $A\subset [n]$ let $\face{i}{A}$ denote the face of $\simplex{m}{n}$ induced by the vertices $\vertex{i}{A}$. 
We are ready to state our main result.

\begin{theorem} \label{t:tv-bound}
Let $d$, $m$, and $p$ be positive integers, with $p$ a prime power, 
and let $n =  \lceil(\frac{d}{m}+1)(p-1) + \frac{1}{m}\rceil$. For any continuous map $f : \simplex{m}{n} \to \mathbb{R}^d$ 
 there exists an integer $i\in [m]$ and a partition $A_1 \cup A_2 \cup \cdots \cup A_p = [n]$ such that
\[ f\big(\face{i}{A_1}\big) \cap
f\big(\face{i}{A_2}\big) \cap \cdots \cap f\big(\face{i}{A_p}\big) \neq \emptyset.\]
\end{theorem}

Observe that for $m=1$ we recover the usual topological Tverberg theorem. The proof of Theorem \ref{t:tv-bound} when $p$ is prime will be given in the following section. In Section 3 we explain how to extend the proof to prime powers. 
Before getting to the proof we make a few remarks and point out some consequences of Theorem \ref{t:tv-bound}. 

\smallskip

\subsection{Remarks}
Theorem \ref{t:tv-bound} is a topological generalization of a recent result by Dobbins, Lee, and the first author \cite{DHL}. Their result was motivated by the colorful Helly theorem and was formulated in terms of Tverberg partitions for families of convex sets that satisfy the ``colorful Helly hypothesis'', but is easily seen to be equivalent to the {\em affine} case of Theorem \ref{t:tv-bound}. While their proof also uses methods from equivariant topology, and therefore requires $p$ to be  prime (power), convexity also plays a crucial role in their proof (in particular the use the hyperplane separation theorem and Sarkaria's tensor method) and is not adaptable to deal with the general continuous case. Let us also note that they showed that when $d$ is divisible by $m$, then the value of $n$ is optimal (even in the affine case) \cite[Remark 3]{DHL}.

\smallskip

The affine case is interesting in its own right due to the following connection to geometric transversal theory.  
A theorem of Montejano \cite{km2011, mont2013} states that if we take three red convex sets $R_1$, $R_2$, $R_3$ and three blue convex sets $B_1, B_2, B_3$ in $\mathbb{R}^3$ where $R_i \cap B_j \neq \emptyset$ for all $i$ and $j$, then there is a {\em line transversal} to the red sets or to the blue sets. 
To see how this result follows from (the affine case of) Theorem \ref{t:tv-bound}, choose a point $p_{i,j} \in R_i \cap B_j$ for every $i$ and $j$. 
Thus the points $p_{i,j}$ are naturally indexed by the vertices of $\simplex{2}{3}$, and there is an affine map $f: \simplex{2}{3} \to \mathbb{R}^3$ 
such that the triangle $f\big( \face{1}{i} \big)$ is contained in $R_i$ and the triangle $f\big( \face{2}{j} \big)$ is contained in $B_j$ for every $i$ and $j$. (This uses the fact that the sets are convex.)
By Theorem \ref{t:tv-bound} there is a partition of the red triangles or the blue triangles into two parts such that their convex hulls intersect. But it is easily seen that such a partition implies that the triangles can be intersected by a single line. (Again this uses the convexity of the sets.)
Since these triangles are contained in distinct sets of the same color, the result follows. For other values of $d$, $m$, and $p$ one can obtain similar types of results on higher dimensional transversals. (See \cite{DHL} and \cite{km2011} for further details.)

\smallskip

Another interesting instance of Theorem \ref{t:tv-bound} is when $m = d+1$. In the case when $f$ is an affine map we get families $F_1$, $F_2$, $\dots$, $F_{d+1}$ of convex sets in $\mathbb{R}^d$, where  
$F_i = \Big\{ f\big( \face{i}{j} \big) \Big\}_{j=1}^n$. 
Since
\[ (j_1, j_2, \dots, j_{d+1}) \in
\face{1}{\{j_1\}} \cap
\face{2}{\{j_2\}} \cap \cdots \cap
\face{d+1}{\{j_{d+1}\}}  \]
for every $(j_1, j_2, \dots, j_{d+1}) \in [n]^{d+1}$, it follows that 
the families $F_1$, $F_2$, $\dots$, $F_{d+1}$ satisfy the hypothesis of the colorful Helly theorem \cite{col-hell}. Therefore there exists a point that intersects every member of one of the families $F_i$.  In other words, there exists an $i\in [d+1]$ such that 
\[\textstyle \bigcap_{j =1}^n f\big( \face{i}{\{j\}}  \big) \neq \emptyset.\]

We do {\em not} claim that this is a consequence of Theorem \ref{t:tv-bound}, however,
when $f$ is an arbitrary continuous map we obtain a ``weak'' colorful topological Helly theorem stated below. Note that this differs from the colorful topological Helly theorem of Kalai and Meshulam \cite{KM} as it does not require that the sets form a good cover, but the conclusion we get is also weaker. 


\begin{corollary} \label{cor:ch}
For any continuous map $f : \simplex{d+1}{n} \to \mathbb{R}^d$  there exists an integer $i \in [d+1]$ and a subset $S\subset [n]$ with $|S|\geq \frac{1}{2d+1}n-o(n)$ such that 
\[\textstyle \bigcap_{j\in S} f\big( \face{i}{\{j\}}  \big) \neq \emptyset.\]
\end{corollary}

\begin{proof}
Suppose first that $n = \big\lfloor \frac{2d+1}{d+1}p \big\rfloor$ for some prime (power) $p$. Since $n\geq \frac{2d+1}{d+1}(p-1)+\frac{1}{d+1}$ it follows from Theorem \ref{t:tv-bound}
that there is an integer $i\in [d+1]$ and a partition $A_1\cup A_2 \cup \cdots \cup A_p = [n]$ such that
\[f \big( \face{i}{A_1} \big) \cap f \big( \face{i}{A_2} \big) \cap \cdots \cap f \big( \face{i}{A_{p}} \big) \neq \emptyset. \] Since the average size of a part $A_i$ is at most $\frac{2d+1}{d+1} < 2$ there are at least $\frac{1}{d+1}p \ge \frac{1}{2d+1}n$ of the parts $A_i$ that have size 1, which is what we wanted to show. The result for general $n$ follows from the fact that  the interval $[n, n+n^{0.525}]$ contains a prime for all sufficiently large $n$ (see \cite{BHP}).
\end{proof}

\section{Proof of Theorem \ref{t:tv-bound}}

\subsection{Notation}
We use standard notions and terminology from topological combinatorics, in particular the \textit{configuration space /  test map} scheme presented for instance in \cite{matousek2003using}. One crucial detail to keep in mind is the distinction between {partitions} and {disjoint unions}. Recall that 
\[A_1 \cup A_2 \cup \cdots \cup A_k = [n]\]
is a {\em partition} if all the sets $A_i$ are nonempty and pairwise disjoint. When we say {\em disjoint union}, on the other hand, we allow for some of the $A_i$ to be empty. For partitions or disjoint unions, the order of the parts does not matter, but in the case when the order of parts plays a role we use the notation
\[A_1 \uplus A_2 \uplus \cdots \uplus A_k.\]

\subsection{Proof outline}
Fix integers $m\geq 1$, $n\geq 2$, and $p\geq 2$, with $p$ prime. We define a simplicial complex $K=K(n,m,p)$ on a vertex set consisting of an ordered union of $p$ disjoint copies of $[n]^m$. That is, 
\[ V(K) = \underset{\text{part 1}}{[n]^m} \uplus \underset{\text{part 2}}{[n]^m} \uplus \cdots \uplus \underset{\text{part $p$}}{[n]^m}. \]
 The facets of $K$ are given by 
\[\face{i}{A_1} \uplus \face{i}{A_2} \uplus \cdots \uplus \face{i}{A_p}\]
for all $1\leq i\leq m$ and for all disjoint unions $A_1\uplus A_2 \uplus \cdots \uplus A_p = [n]$. Here it should be understood that $\face{i}{A_j}$ belongs to the $j$th part of the vertex set. Observe that the cyclic group $\mathbb{Z}_p$ acts on $K$ by cyclically shifting the parts of the union. Thus $K$ is a $\mathbb{Z}_p$-invariant subcomplex of $\big( \simplex{m}{n} \big)^{*p}_{\Delta(2)}$, the $p$-fold pairwise deleted join of $\simplex{m}{n}$. (The complex $K$ is the so-called ``configuration space'' for our problem.)

\smallskip

There are two main steps to the proof of Theorem \ref{t:tv-bound}. The first step is to show that if the conclusion of the theorem fails for some continuous map $f : \simplex{m}{n} \to \mathbb{R}^d$, then we can construct a $\mathbb{Z}_p$-map (the so-called ``test map'') 
\[ K \to  S^{(d+1)(p-1)-1} \]

\smallskip

The second step is to  construct a $\mathbb{Z}_p$-complex $T$ together with a
$\mathbb{Z}_p$-map  $T \to K$. The 
complex $T$ will be a subcomplex of the nerve complex of the facets of $K$. We will show that when $n$ satisfies the bound in the hypothesis of Theorem \ref{t:tv-bound}, then the complex $T$ is $\big( (d+1)(p-1)-1 \big)$-connected.
Therefore, a counter-example to Theorem \ref{t:tv-bound} would give us $\mathbb{Z}_p$-map  $T \to S^{(d+1)(p-1)-1}$, contradicting Dold's Theorem.  

\begin{theorem}[Dold \cite{dold}]
    Let $G$ be a finite, nontrivial group, $X$ an $n$-connected $G$-space, and $Y$ a paracompact topological space with a free $G$-action and dimension at most $n$. Then there does not exist a $G$-map $X\rightarrow Y$.
\end{theorem}

\subsection{Proof}
Since our configuration space $K$ is a subcomplex of $\big(  \simplex{m}{n} \big)^{*p}_{\Delta(2)}$,  the points of $K$ may be represented as {\em formal convex combinations}
\[ t_1x_1\oplus t_2 x_2 \oplus \cdots \oplus t_px_p \]
where 
\begin{itemize}
\item $t_1+\cdots+t_p = 1$ and $t_i \geq 0$ for all $i$, 
\item each $x_i$ is in the simplex spanned by the $i$th part of $V(K)$, and 
\item  $\textrm{supp}(x_1) \cup \cdots \cup \textrm{supp}(x_p)$ is a face of $K$, i.e., there is some $i$ and a disjoint union $A_1 \cup \cdots \cup A_p = [n]$ such that $\textrm{supp}(x_j) \subset \face{i}{A_j}$ for all $1\leq j\leq p$.
\end{itemize}
Note that there is a natural action of $\symg{p}$, the symmetric group on $p$ elements, on $K$ given by
\begin{multline*}
\pi\: (t_1x_1\oplus t_2x_2 \oplus \cdots \oplus t_px_p) = \\ t_{\pi^{-1}(1)}x_{\pi^{-1}(1)} \oplus t_{\pi^{-1}(2)}x_{\pi^{-1}(2)} \oplus  \cdots \oplus t_{\pi^{-1}(p)}x_{\pi^{-1}(p)}.
\end{multline*}
Furthermore, the subgroup of $\symg{p}$ generated by the permutation $(1\, 2\,\cdots\,p)$ (written in cycle notation) induces a free $\mathbb{Z}_p$-action on $K$.

\smallskip

Recall that $\big( \mathbb{R}^d \big)^{*p}_\Delta$, the $p$-fold ($p$-wise) deleted join of $\mathbb{R}^d$ is a subset of the $p$-fold join of $\mathbb{R}^d$, and is given by
\[\textstyle
\big( \mathbb{R}^d \big)^{*p}_\Delta = \big( \mathbb{R}^d \big)^{*p} \setminus \left\{\frac{1}{p}x \oplus \frac{1}{p}x \oplus \cdots \oplus \frac{1}{p}x : x\in \mathbb{R}^d\right\},
\]
Here $\symg{p}$ acts on $\big( \mathbb{R}^d \big)^{*p}_\Delta$ in a natural way by permuting the parts, and in particular,  $\big( \mathbb{R}^d \big)^{*p}_\Delta$ has  a free $\mathbb{Z}_p$-action when $p$ is prime.

\smallskip

Let $f:\simplex{m}{n} \rightarrow \mathbb{R}^d$ be a continuous map as in Theorem \ref{t:tv-bound}, and assume for contradiction that 
\[f\big( \face{i}{A_1} \big) \cap f\big( \face{i}{A_2} \big) \cap \cdots \cap f\big( \face{i}{A_p} \big) = \emptyset \]
for every $i \in [m]$ and every partition $A_1 \cup A_2 \cup \cdots \cup A_p = [n]$. 
This assumption allows us to define the test map $\tilde{f}: K \rightarrow \big( \mathbb{R}^d \big)^{*p}_\Delta$  by
\[
t_1x_1\oplus t_2x_2 \oplus \cdots \oplus t_px_p \mapsto t_1f(x_1) \oplus t_2f(x_2) \oplus \cdots \oplus t_pf(x_p),
\] 
which is a $\mathbb{Z}_p$-map. Since there exists a $\mathbb{Z}_p$-map $\big( \mathbb{R}^d \big)^{*p}_\Delta \rightarrow S^{(d+1)(p-1)-1}$ where $S^{(d+1)(p-1)-1}$ is equipped with a free $\mathbb{Z}_p$-action (see e.g. \cite[Proposition 6.3.2]{matousek2003using}), we obtain the following.

\begin{lemma} \label{lem:test}
If $f\big( \face{i}{A_1} \big) \cap f\big( \face{i}{A_2} \big) \cap \cdots \cap f\big( \face{i}{A_p} \big) = \emptyset$ for every $i\in [m]$ and every partition $A_1\cup A_2 \cup \cdots \cup A_p = [n]$, then there exists a $\mathbb{Z}_p$-map $K \to S^{(d+1)(p-1)-1}$.
\end{lemma}

\smallskip

We now want to show that if $n\geq \lceil(\frac{d}{m}+1)(p-1) + \frac{1}{m}\rceil$, then there exists a $((d+1)(p-1)-1)$-connected $\mathbb{Z}_p$-complex $T$ and a $\mathbb{Z}_p$-map
$T \to K$, which gives us the desired contradiction to Dold's theorem. The construction of $T$ and the $\mathbb{Z}_p$-map requires several steps and takes up the remainder of this section.

\smallskip

Let $F$ denote the set of facets of $K$, and let $N$ be the \textit{nerve} of $F$, that is, the simplicial complex with vertex set $F$  whose faces are the subsets $S \subset F$ such that $\bigcap_{\tau \in S} \tau \neq \emptyset$. 
Then $\symg{p}$ acts on $N$ by 
\begin{multline*}
\pi\: \big(\face{i}{A_1} \uplus \face{i}{A_2} \uplus \cdots \uplus \face{i}{A_p} \big) = \\ 
\face{i}{A_{\pi^{-1}(1)}} \uplus \face{i}{A_{\pi^{-1}(2)}} \uplus \cdots \uplus \face{i}{A_{\pi^{-1}(p)}}.
\end{multline*}
Note that the action of the subgroup $\mathbb{Z}_p$ on $K$ is free, which gives us a free 
$\mathbb{Z}_p$-action on $N$.

\begin{lemma} \label{lem:NerveMap}
There is a $\mathbb{Z}_p$-map $N \to K$.
\end{lemma}

\begin{proof}
Recall that for a simplicial complex $L$, its barycentric subdivision $\bary{L}$ is the simplicial complex whose vertices are the nonempty faces of $L$ and simplices are the chains of nonempty faces of $L$ ordered by inclusion. Note that if $L$ is a $\mathbb{Z}_p$-complex, then the group action on $L$ induces a  group action on $\bary{L}$, and for this induced group action there are $\mathbb{Z}_p$-maps $\bary{L} \to L$ and $L \to \bary{L}$. 
Therefore it suffices to exhibit a $\mathbb{Z}_p$-map $\bary{N} \to \bary{K}$.

Let $\tau = \{\tau_1,\dots,\tau_k\}$ be a nonempty face of $N$. 
Observe that $\tau$ is a vertex of $\bary{N}$ and $\bigcap_{i=1}^k\tau_i$ is a vertex of $\bary{K}$. 
Moreover, for the chain 
\[ \{\tau_1\} \subset \{\tau_1, \tau_2\} \subset \cdots \subset \{\tau_1, \tau_2, \dots, \tau_k\},
\]
we have
\[ \tau_1 \supset (\tau_1\cap \tau_2) \supset \cdots \supset (\tau_1 \cap \tau_2 \cap \cdots \cap \tau_k) \]
and therefore $\tau \mapsto \bigcap_{i=1}^k \tau_i$ defines a simplicial map $\bary{N} \to \bary{K}$. Since  $\pi \hspace{.2ex} \tau = \{\pi \hspace{.2ex} \tau_1, \pi \hspace{.2ex} \tau_2, \dots, \pi \hspace{.2ex} \tau_k\}$ and $\pi \hspace{.2ex} \bigcap_{i=1}^k \tau_i = \bigcap_{i=1}^k \pi \hspace{.2ex}  \tau_i$, this gives us the desired $\mathbb{Z}_p$-map.
\end{proof}

\smallskip

We now define a family of vertex disjoint, $\mathbb{Z}_p$-invariant subcomplexes of $N$. Our goal is to show that their join gives us a highly connected subcomplex of $N$. 

\smallskip

Let $M$ be the {\em induced subcomplex} of $N$ on the vertex set 
\[ \big\{ \face{1}{A_1} \uplus \face{1}{A_2} \uplus \cdots \uplus \face{1}{A_p} \big\} \]
taken over all {\em disjoint unions} $A_1 \uplus A_2 \uplus \cdots \uplus A_p = [n]$. Clearly $M$ is a $\mathbb{Z}_p$-invariant subcomplex of $N$. We also observe that $M$ has precisely $n\cdot p$ facets. Indeed for every $j \in [n]$ and $k \in [p]$ we get a unique facet induced by the vertices of $M$ for which $j\in A_k$.

\begin{lemma}\label{lem:M}
    The complex $M$ is  $(n-2)$-connected.
\end{lemma}
\begin{proof}
Recall that the $p$-fold pairwise deleted join of the $(n-1)$-dimensional simplex $\big( \Delta^{n-1} \big)^{*p}_{\Delta(2)}$ is an $(n-2)$-connected space (see e.g. \cite[Corollary 6.4.4]{matousek2003using}). 
Let $B$ denote the nerve complex of the set of facets of $\big( \Delta^{n-1} \big)^{*p}_{\Delta(2)}$. By the nerve theorem $B$ is homotopy equivalent to $\big( \Delta^{n-1} \big)^{*p}_{\Delta(2)}$, and it follows that $B$ is also $(n-2)$-connected. 

By viewing $\Delta^{n-1}$ as the simplex on the vertex set $[n]$, a facet of $\big(\Delta^{n-1}\big)^{*p}_{\Delta(2)}$  (i.e. a vertex of $B$) can be written as  the disjoint union
\[ A_1 \uplus A_2 \uplus \cdots \uplus A_p = [n]. \]   
To complete the proof we simply observe that the map $V(B) \rightarrow V(M)$ defined by
\[ A_1 \uplus A_2 \uplus \cdots \uplus A_p \mapsto 
\face{1}{A_1} \uplus  \face{1}{A_2} \uplus \cdots \uplus \face{1}{A_p} \]
is a simplicial isomorphism between $B$ and $M$.
\end{proof}

For each $2\leq i\leq m$,  define a subcomplex of $L_i \subset N$ on the set of vertices
\[ \big\{ \face{i}{A_1} \uplus \face{i}{A_2} \uplus \cdots \uplus \face{i}{A_p}  \big\} \]
taken over all partitions $A_1 \uplus A_2 \uplus  \cdots \uplus A_p = [n]$. The faces of $L_i$ are the subsets
\[ \big\{ \face{i}{A_{j,1}} \uplus \face{i}{A_{j,2}} \uplus \cdots \uplus \face{i}{A_{j,p}} \big\}_{j=1}^k\]
where $\bigcap_{j=1}^k A_{j,\ell} \neq \emptyset$ for every $1\leq \ell \leq p$.

Clearly $L_i$ is a $\mathbb{Z}_p$-invariant subcomplex of $N$. We also observe that $L_i$ has precisely $\frac{n!}{(n-p)!}$ facets.  
Indeed, for every ordered $p$-tuple $(j_1, j_2, \dots, j_p)$ of $[n]$ we get a unique facet induced by the vertices of $L_i$ for which $j_1\in A_1$, $j_2\in A_2$, \dots, $j_p\in A_p$.

\begin{lemma}\label{lem:Li}
    For all $2\leq i\leq m$, the complex $L_i$ is $(n-p-1)$-connected.
\end{lemma}

\begin{proof}
Let $C_{n,p}$ be the simplicial complex whose vertices are the ordered partitions of $[n]$ into $p$ non-empty parts $(A_1,\dots,A_p)$ with faces
\[ \{(A_{j,1}, A_{j,2}, \dots,A_{j,p})\}_{j=1}^k \]
where $\bigcap_{j=1}^k A_{j,\ell} \neq \emptyset$ for all $1 \leq \ell \leq p$. It is straightforward to show that $C_{n,p}$ and $L_i$ are isomorphic, and it was shown in \cite[Lemma 2.1]{DHL} that $C_{n,p}$ is 
$(n-p-1)$-connected. 
\end{proof}

Recall that the vertices of $N$ are the facets of $K$, and it follows from the definition that a facet of $K$ is a vertex in at most one of the complexes $M, L_2, \dots, L_m$ (but not every vertex of $N$ appears in one of these subcomplexes). Since the complexes $M, L_2, \dots, L_m$ are vertex disjoint we may take their join and define
\[T = M * L_2 * \cdots * L_m.\]

\begin{lemma} \label{lem:T}
  The complex  $T$ is a $\mathbb{Z}_p$-invariant subcomplex of $N$. 
\end{lemma}    
\begin{proof}
We first show that $T\subset N$. The $\mathbb{Z}_p$-invariance follows from the fact that $M, L_2, \dots, L_p$ are all $\mathbb{Z}_p$-invariant subcomplexes of $N$.

Let $\tau=\tau_1 \cup \tau_2 \cup \cdots \cup \tau_m$ be a facet of $T$, where $\tau_1$ is a facet of $M$ and $\tau_i$ is a facet of $L_i$ for all $2\leq i \leq m$.
Since each $\tau_i$ is a collection of facets of $K$, it follows that $\tau$ is a face of $N$ if and only if there exists a vertex  
\[v \in \underset{\text{part 1}}{[n]^m} \uplus \underset{\text{part 2}}{[n]^m} \uplus \cdots \uplus \underset{\text{part }p}{[n]^m} = V(K)\] which belongs to every facet (of $K$) in 
$\tau_1 \cup \tau_2 \cup \cdots \cup \tau_m$.

To see why this is true we first observe that for every $2\leq i \leq p$ there is an ordered $p$-tuple $(j_{i,1}, j_{i,2}, \dots, j_{i,p})$ of $[n]$ such that every facet (of $K$) in $\tau_2 \cup \cdots \cup \tau_m$ contains 
\[ \renewcommand{\arraystretch}{1.25}
\begin{array}{cccccccl}
  [n] & \uplus & \{j_{2,1}\} & \uplus & \cdots & \uplus &  \{j_{m,1}\} & \text{  in part 1,} \\
  
  [n] & \uplus & \{j_{2,2}\} & \uplus & \cdots & \uplus &  \{j_{m,2}\} & \text{  in part 2,} \\
  
  &&&&\vdots &&& \\
  
  [n] & \uplus & \{j_{2,p}\} & \uplus & \cdots & \uplus &  \{j_{m,p}\} & \text{  in part $p$.} 
\end{array}
\] 
Next, observe that there exists integers $j \in [n]$ and $k\in [p]$ such that every facet (of $K$) in $\tau_1$ contains
\[
\begin{array}{cccccccl}
  \{j\} & \uplus & [n] & \uplus & \cdots & \uplus &  [n] & \text{  in part $k$.} \\
\end{array}
\]
It follows  that  the vertex $v = (j,j_{2,k}, \dots, j_{m,k})$ in the $k$th part of $V(K)$ belongs to every facet (of $K$) in $\tau_1 \cup \tau_2 \cup \cdots \cup \tau_m$.  
\end{proof}

\smallskip

In summary we have the following $\mathbb{Z}_p$-maps
\[T \underset{\text{Lemma } \ref{lem:T}}{\hookrightarrow} N \underset{\text{Lemma }\ref{lem:NerveMap}} \to K \underset{\text{Lemma } \ref{lem:test}}{\to} S^{(d+1)(p-1)-1}. \]
To complete the proof of Theorem \ref{t:tv-bound},  observe that for $n\geq \lceil  (\frac{d}{m} + 1)(p-1) + \frac{1}{m} \rceil$ we have
\[n-2 + (m-1)(n-p+1) \geq (d+1)(p-1)-1.\]
By the connectivity of the join (see e.g. \cite[Proposition 4.4.3]{matousek2003using}) together with Lemmas \ref{lem:M} and \ref{lem:Li} it follows that $T$ is $\big( (d+1)(p-1)-1 \big)$-connected, which gives us the desired contradiction to Dold's theorem. \qed

\section{Final remarks}

The proof of Theorem \ref{t:tv-bound} in the case when $p = q^r$ is a prime power is more or less identical to the prime case. The set up of the configuration space and the test map is the same, and even though the symmetric group $\symg{p}$ 
does not act freely, it does give a {\em fixed-point free} action on $K$ and on the target space $S^{(d+1)(p-1)-1}$. Therefore the action of the subgroup $\mathbb{Z}_q\times \mathbb{Z}_q \times \cdots \times \mathbb{Z}_q \subset \symg{p}$ is also fixed-point free, which  gives us an equivariant test-map that contradicts the following theorem of Volovikov \cite{volov96}.

\begin{theorem*}[Volovikov]
Let $G = \mathbb{Z}_q \times \mathbb{Z}_q \times \cdots \times \mathbb{Z}_q$ be the product of finitely many copies with $q$ prime. Let $X$ and $Y$ be fixed point free $G$-spaces, where $X$ is $n$-connected and $Y$ is finite-dimensional and homotopy equivalent to $S^n$. Then there is no $G$-equivariant map $X\to Y$.   
\end{theorem*}

\smallskip

Regarding Corollary \ref{cor:ch}, we do not have any reason to believe that the constant $\frac{1}{2d+1}$ is optimal, but it seems that our proof method can not lead to further improvements. For $d\geq 2$ we do not even have an example that shows that the constant is strictly less than 1. (For $d=1$ it is easily seen that the constant  equals 1.)  
We leave it as an open problem to determine the optimal constant.

\end{document}